[1994/06/01]

\documentclass{amsart}
\usepackage{amsthm}
\usepackage{amsmath,amsxtra,amssymb,latexsym, amscd,amsthm}
\usepackage[all]{xy}
\usepackage{mathrsfs}
\usepackage{enumerate}
\usepackage{color}
\usepackage{multicol}
\newtheorem{thm}{Theorem}[section]
\newtheorem{cor}[thm]{Corollary}
\newtheorem{prop}[thm]{Proposition}
\newtheorem{lem}[thm]{Lemma}
\theoremstyle{definition}
\newtheorem{defn}{Definition}[section]
\newtheorem{exam}{Example}[section]

 \theoremstyle{remark}
\newtheorem{rmk}{Remark}[section]
\newtheorem{\theequation}{}[section]

\renewcommand{\theequation}{\thesection.\arabic{equation}}
\usepackage{amsfonts}

\newcommand{\C}{\field{C}}
\newcommand{\R}{\field{R}}

\def\2ovlOB{\overline {\overline \Omega}_B}
\def\1ovlO{\overline \Omega}

\def\R{{\Bbb R}}
\def\BB{{\Bbb B}}
\def\C{{\Bbb C}}

\def\B{{\mathcal B}}

\def\N{{\Bbb N}}

\makeatletter
\@namedef{subjclassname@2020}{%
	\textup{2020} Mathematics Subject Classification}
\makeatother
\begin{document}

\title[EXTENDED CES\`ARO COMPOSITION OPERATORS ON  WEAK BLOCH-TYPE SPACES]{EXTENDED CES\`ARO COMPOSITION OPERATORS ON   WEAK BLOCH-TYPE SPACES ON THE UNIT BALL OF A HILBERT SPACE}

\author{Thai Thuan Quang} 
\address{Department of Mathematics and Statistics, 
Quy Nhon University,
170 An Duong Vuong, Quy Nhon, Binh Dinh, Vietnam.}
\email{thaithuanquang@qnu.edu.vn} 


%
%
\date{March 10, 2022.}

\maketitle

\bigskip
\begin{abstract} Denote by $ B_X $   the unit ball of  an infinite-dimensional  complex Hilbert space $ X. $ Let $\psi \in H(B_X),$ the space of all holomorphic functions on the unit ball $B_X,$   $\varphi    \in S(B_X)$ the set of holomorphic self-maps of $B_X. $ Let $C_{\psi, \varphi}: \mathcal B_{\nu}(B_X)$ (and $    \mathcal B_{\nu,0}(B_X)$) $\to \mathcal B_{\mu}(B_X)  $ (and $ \mathcal B_{\mu,0}(B_X)$) be weighted extended Ces\`aro operators   induced
	by products of  the extended Ces\`aro operator $ C_\varphi $ and integral  operator $T_\psi.$  
	In this paper, we characterize the boundedness and compactness of $ C_{\psi,\varphi} $ via  the estimates     for the restrictions of   $ \psi $ and $ \varphi $ to  a $ m$-dimensional subspace of $ X $   for some $ m\ge2. $ Based on these we give   necessary as well as  sufficient conditions for the boundednees, the (weak) compactness of   $ \widetilde{C}_{\psi, \varphi} $ between 
	spaces of Banach-valued holomorphic functions  weak-associated to $ \mathcal B_{\nu}(B_X) $ and $ \mathcal B_{\mu}(B_X). $
	\end{abstract}

\bigskip
\section{Introduction}


Let $ E $ be a space of holomorphic functions on the unit ball  $B_X $  in a complex Banach space $ X.$  Let $ H(B_X) $ be the
class of all holomorphic functions on $ B_X $ and $ S(B_X) $ the collection of all the holomorphic self-maps of $B_X. $  For a  $ \varphi \in S(B_X)$  and a  $ \psi \in H(B_X),$ the composition operator $ C_\varphi $ and the extended Ces\`aro operato $ T_\psi $ are defined by
\[ (C_\varphi f)(z) := (f \circ \varphi)(z),\quad (T_\psi f)(z) := \int_0^1f(tz)R\psi(tz)\dfrac{dt}{t} \quad \forall f \in E, \ \forall z \in \BB_n.\]
where $ R\psi(z):= D\psi(z)z $ is  the radial derivative of $ \psi$ at $ z$ with $ D\psi(z) $ is the Fr\'echet derivative of $ \psi $ at $ z. $ 

The study of composition operators consists in the comparison of the properties of the $ C_\varphi $  with that of the function $ \varphi $ itself, which is called the symbol of $ C_\varphi. $ One can characterize boundedness and compactness of $ C_\varphi $ and many other properties. 

The problem of studying of composition operators on various Banach spaces of holomorphic functions
on the unit disk or the unit ball, such as Hardy and Bergman spaces, the space $ H^\infty $  of all
bounded holomorphic functions, the disk algebra and weighted Banach spaces with sup-norm, etc. received a special attention of many authors during the past several decades. The weighted composition operators on
these spaces appeared in some works   with different applications. There is a great number of topics on operators of such a type: boundedness and compactness, compact differences, topological structure, dynamical and ergodic properties.

The extended Ces\`aro operator $ T_\psi $ is a natural extension of the Ces\`aro operator acting $ f \in H(\BB_1) $ where $ \BB_1 $ is the unit ball in $ \C, $
\[ C[f](z) := \sum_{j=0}^\infty\Bigg(\dfrac{1}{j+1}\sum_{k=0}^ja_k\Bigg)\]
with $ f(z) = \sum_{j=0}^\infty a_jz^j, $ the Taylor expansion of $ f. $

It is well know that $ C[\cdot]  $ acts as a bounded linear operator on various spaces of holomorphic functions,  including the Hardy and Bergman spaces. But it is not bounded on the Bloch space (see \cite{Xi}). A little calculation shows 
\[ C[f](z) = \frac{1}{z}\int_0^zf(\eta)\frac{1}{1-\eta}d\eta =    \int_0^1f(tz)\bigg(\log\dfrac{1}{1-\eta}\bigg)'\Big|_{\eta=tz}dt.\]
Hence, on most holomorphic function spaces, $ C[\cdot] $ is bounded if and only if the integral operator $ f\mapsto    \int_0^zf(t)\big(\log\frac{1}{1-t}\big)'dt $
is bounded. From this point of view it is natural to consider the extended Ces\`aro operator with holomorphic symbol $ \psi,$
\[ T_\psi f(z)  := \int_0^zf(\eta)\psi'(\eta)d\eta . \]
The boundedness and compactness of this operator on Hardy spaces, Bergman spaces, Bloch- type spaces and Lipschitz spaces have been studied in \cite{AC, ASi, WH}.
The extended Ces\`aro operator is a generalization of this operator. It has been well studied in many papers, see, for example, \cite{ASt,LS1,Hu} as well as the related references therein.

It is natural to discuss the product of extended Ces\`aro operator and composition operator. For $ \varphi \in S(\BB_n)$ and $ \psi \in H(\BB_n) $ the product can be expressed as
\begin{equation}\label{C_operator}
C_{\psi,\varphi}(f) := T_\psi C_\varphi f(z) = \int_0^1f(\varphi(tz))R(\psi(tz))\dfrac{dt}{t} \quad f \in H(\BB_n), z \in \BB_n.
\end{equation}
This operator is called \textit{extended  Ces\`aro composition operator}. It is interesting to characterize the boundedness and compactness of the product operator on all kinds of function spaces. Even on the disk of
$ \C, $ some properties are not easily managed;
see some recent papers in \cite{LS2,LS3,Ya,LZ, Ta}. 


Building on those foundations, the present paper continues this line of research and discusses the operator in infinite dimension. Namely, we study the boundedness and the compactness of $ C_{\psi,\varphi}: E_1 \to E_2 $ between weighted Bloch-type spaces of holomorphic functions on the unit ball of a Hilbert space, which has been investigated by T. T. Quang in \cite{Qu}. Based on these results we give the characterizations of the boundedness and the compactness of the operators $ \widetilde{C}_{\psi,\varphi}: WE_1(Y) \to WE_2(Y) $ between 
 spaces of   holomorphic functions with values in a Banach space $ Y, $ weak-associated to    $ E_1, E_2 $ in the following sense: 
 
%

Let $ E $ be a   space of  holomorphic functions on the unit ball $ B_X $ of a Banach space $ X$ such that it contains the constant functions and 
 its closed unit ball $ B_E $ is compact in the compact open topology $ \tau_{co} $
of $ B_X.$  Let $ Y $ be a Banach space and $ W\subset Y' $ be a separating subspace of the dual $ Y' $ of $ Y.$ 
 We say that the space    
\[ WE(Y) := \{f: B_X \to Y:\  f \ \text{is locally bounded and}\ w\circ f \in E, \ \forall w \in W\}
\] 
equipped with the norm
\[ \|f\|_{WE(Y)} := \sup_{w\in W,\|w\|
	\le 1}\|w\circ f\|_E.  \]
is the  $ Y$-valued Banach space $ W$-associated to $ E. $ 


%




The paper is organized as follows. 

Section 2 is devoted to recall  some fundamental properties of the    Banach  space of Banach-valued holomorphic functions $ W$-associated to   a Banach space of scalar-valued holomorphic functions on the closed unit ball  of a Banach space. We also  introduce some assumptions  on the subspace $ W $ of the dual $ Y' $ of $ Y $ to give  necessary as well as  sufficient conditions for the boundednees, the (weak) compactness of   $ \widetilde{C}_{\psi, \varphi} $ via the respective properties of 
$C_{\psi, \varphi} $ (Theorem \ref{thm_WO}). Some   fundamental results on weighted Bloch-type space of holomorphic functions on the unit ball of a Hilbert space are also   summarized in this section.

The main results are stated in Section 3. Via the estimates for the restrictions of the functions $ \psi $ and $ \varphi $ to $ \BB_m $ for some $ m\ge2, $  we characterize the boundedness and the  compactness  of the operators  $C_{\psi, \varphi}$  between the spaces of (little) Bloch-type  $ \mathcal B_\mu(B_X), $ $ \mathcal B_{\mu,0}(B_X)$ as well as   the equivalent relationships between them. 
It should be noted that a necessary condition (but not sufficient) and a sufficient condition (but not necessary) for the compactness of $ C_{\psi, \varphi} $ are also obtained after any necessary minor modifications  for the holomorphic self-map $ \varphi. $ 
Finally, by applying the results we give the necessary as well as sufficient conditions for the boundedness and the (weak) compactness of $ \widetilde{C}_{\psi,\varphi}. $
Several helpful test functions and  auxiliary results concerning our computations will be introduced in this section. 

The proofs of the main theorems above will be given in Sections 4 and . 

Throughout this paper, we use the notions $X \lesssim Y$ and $X \asymp Y$ for non negative quantities $X$ and $Y$ to mean $X \le CY$ and, respectively, $Y/C \le X \le CY$ for some inessential constant $C >0.$

 \section{Preliminaries and Auxiliary Results}
 \setcounter{equation}{0}

 \subsection{Weak holomorphic spaces}
Let  $X, Y$ be  complex     Banach spaces. Denote by $ B_X $ the  unit ball  of $ X $ (we write $ \BB_n $ instead of $ B_{\C^n} $). 

 By $H(B_X, Y)$ we denote the vector  space of $Y$-valued holomorphic functions on $B_X.$ A holomorphic function $f \in H(B_X, Y)$ is called locally bounded holomorphic on $B_X$ if for every $z \in B_X$ there exists a neighbourhood $U_z$ of $0\in X$ such that $f(U_z)$ is bounded. Put
$$H_{LB}(B_X, Y) = \big\{f \in H(B_X, Y): \quad f \; \text{is locally bounded on}\; B_X\big\}.$$

 Suppose that $ E $ is a Banach space of holomorphic functions $ B_X \to \C $ such that
\begin{enumerate}
	\item [(e1)] $ E $ contains the constant functions,
	\item [(e2)] the closed unit ball $ B_E $ is compact in the compact open topology $ \tau_{co} $
	of $ B_X.$
\end{enumerate}
It is easy to check that the properties (e1)-(e2) are satisfied by a large number of well-known function
spaces, such as  classical  Hardy, Bergman, BMOA, and Bloch spaces.

Let $ W\subset Y' $ be a separating subspace of the dual $ Y' $ of $ Y.$ Consider the $ Y$-valued Banach space $ W$-associated to $ E $ 
	\[ WE(Y) := \{f: B_X \to Y:\  f \ \text{is locally bounded and}\ w\circ f \in E, \ \forall w \in W\}
\] 
equipped with the norm
\[ \|f\|_{WE(Y)} := \sup_{w\in W,\|w\|
	\le 1}\|w\circ f\|_E.  \]

%
The following proposition introduce a relation between $ E $ and $ WE(Y): $
\begin{prop}[\cite{Qu}, Proposition 2.1]\label{Prop2.1} Let $X, Y$ be   complex  Banach spaces and $W \subset Y'$ be a separating subspace. Let $ E $ be a Banach space of holomorphic functions $ B_X \to \C $ satisfying (e1)-(e2) and $ WE(Y) $ be the $ Y$-valued Banach space $ W$-associated to $ E. $ Then, the following  assertions  hold:
\begin{enumerate}
	\item [\rm (we1)] $ f\mapsto f \otimes y $ defines a bounded linear operator $ P_y: E \to WE(Y) $ for any $ y \in Y, $ where $ (f\otimes y)(z) = f(z)y $ for $ z \in B_X, $
	\item [\rm (we2)] $ g \mapsto w \circ g $ defines a bounded linear operator $ Q_{w}: WE(Y) \to E $ for any $ w \in W, $
	\item [\rm (we3)] For all $ z \in B_X $ the point evaluations $ \widetilde\delta_z: WE(Y) \to (Y,\sigma(Y,W)), $ where $ \widetilde\delta_z(g) = g(z), $   are continuous.
\end{enumerate}
In the case the hypothesis ``separating'' of $ W $ is replaced by a stronger one that $ W $ is ``almost norming'', we obtain the  assertion (we3') below instead of (we3):
\begin{enumerate}
	\item [\rm (we3')] For all $ z \in B_X $ the point evaluations $ \widetilde\delta_z: WE(Y) \to Y, $     are bounded.
\end{enumerate}
\end{prop}
Here, the subspace $ W $ of $ Y' $ is called \textit{almost norming} if 
\[ q_W(x) := \sup_{w \in W, \|w\|\le 1}|w(x)| \]
defines an equivalent norm on $ Y. $

\begin{thm}[Linearization, \cite{Qu}, Theorem 2.2]\label{thm_1.1} Let $X, Y$ be   complex  Banach spaces and $W \subset Y'$ be a separating subspace. Let $ E $ be a Banach space of holomorphic functions $ B_X \to \C $ satisfying (e1)-(e2). Then there exist a Banach space $^*\!E$ and a mapping $\delta_X \in H(B_X, {^*\!E})$ with the following  universal property:  A function $f \in WE(Y)$ if and only if there is a unique mapping $T_f \in L(^*\!E, Y)$ such that $T_f \circ \delta_X = f.$ This property characterize $^*\!E$ uniquely up to an isometric isomorphism.
	
	Moreover, the mapping
	$$ \Phi: f \in WE(Y) \mapsto T_f  \in L(^*\!E, Y)  $$
	is a topological isomorphism.
\end{thm}

\subsection{The  extended Ces\`aro composition  operators}
\setcounter{equation}{0}
%

Given $ \psi \in H(B_X) $ and $ z \in B_X. $ Then,   we denote by $ D\psi(z) $ the Fr\'echet derivative of $ \psi $ at $ z $ and $ R\psi(z):= D\psi(z)z $  the radial derivative of $ \psi$ at $ z. $ Note that, in the case $ X $ is a complex Hibert space with  an orthonormal basis$ (e_k)_{k \in \Gamma}$ we have 
\[ R\psi(z) = \langle\nabla\psi(z), \overline{z}\rangle \]
where $ \nabla\psi(z) = \Big(\frac{\partial\psi}{\partial z_k}(z)\Big)_{k\in \Gamma} $ is the gradient of $ \psi$ at $ z, $ and  $ \overline{z} = \sum_{k \in \Gamma}\overline{z_k}e_k $ for every $ z = \sum_{k \in \Gamma}z_ke_k \in B_X. $ 
It is obvious that $ |R\psi(z)| \le |\nabla\psi(z)|\|z\| <  |\nabla\psi(z)| $ for every $ z \in B_X. $

Let $ E_1 $ and $ E_2 $ be  Banach spaces of holomorphic functions $ B_X \to \C $ satisfying the conditions (e1) and (e2).  Let $\psi \in H(B_X)$ and  $\varphi \in S(B_X),$  the set of holomorphic self-maps of $B_X.$  Consider the operators $ C_{\psi,\varphi}: E_1 \to E_2 $ and $ \widetilde{C}_{\psi,\varphi}: WE_1(Y) \to WE_2(Y)$ given by
\[\begin{aligned}
C_{\psi,\varphi}(f) &= \int_0^1f(\varphi(tz))R(\psi(tz))\dfrac{dt}{t} \quad f \in E_1, z \in B_X, \\
\widetilde{C}_{\psi,\varphi}(g) &=   \int_0^1g(\varphi(tz))R(\psi(tz))\dfrac{dt}{t} \quad g \in WE_1(Y), z \in B_X.
\end{aligned} \]

Now, let $ y\in Y, $   $ w\in W $ such that $ \|y\| = \|w\| = 1 $ and $ w(y) = 1. $ Consider the maps $ P_y $ and $ Q_w $ as in  Proposition \ref{Prop2.1}. It is easy to check that
\[ C_{\psi,\varphi} = Q_w \circ \widetilde{C}_{\psi, \varphi} \circ P_y. \]
Then, by an argument analogous to that used for the proof of Proposition 2.1 in \cite{Qu} we get the following:
\begin{thm}\label{thm_WO}
 Let $X, Y$ be   complex  Banach spaces and $W \subset Y'$ be a  subspace. Let $ E_1 $ and $ E_2 $ be Banach spaces of holomorphic functions $ B_X \to \C $ satisfying (e1)-(e2).  Let $\psi \in H(B_X)$ and  $\varphi \in S(B_X).$  
 \begin{enumerate}
 	\item[\rm (1)] If $ W $ is separating then $ C_{\psi,\varphi} $    is bounded if and only if $ \widetilde{C}_{\psi,\varphi}$ is bounded;
 	\item[\rm (2)] If $ W $ is almost norming and   $C_{\psi,\varphi}    $ compact then:
 	\begin{itemize}
 		\item[\rm (a)]  $\widetilde{C}_{\psi,\varphi} $ is  compact if and only if the identity map $ I_Y: Y \to Y $ is   compact,     i.e. $ \dim Y <\infty;$
 			\item[\rm (b)]  $\widetilde{C}_{\psi,\varphi} $ is weakly compact if and only if the identity map $ I_Y: Y \to Y $ is weakly compact.
 	\end{itemize}
 \end{enumerate}
\end{thm}
\begin{rmk} In the case $ W $ is separating and $ Y $ is separable, we have $ W $ is almost norming and the identity operator $ I_Y $ is weakly compact. Then as in (b) we get $ \widetilde{C}_{\psi,\varphi} $ is weakly compact  if $ C_{\psi,\varphi} $ is compact.  
	\end{rmk}
\subsection{The Bloch-type spaces on the unit ball of a Hilbert space}
\setcounter{equation}{0}

Throughout the forthcoming, unless otherwise specified, we shall denote by $X$   a complex Hilbert space with the open unit ball $ B_X $ and $ Y $ a Banach space. Denote
\[ H^\infty(B_X, Y) =  \Big\{f \in H(B_X,Y): \ \sup_{z \in B_X}\|f(z)\| <\infty\Big\}. \]

It is easy to check that $ H^\infty(B_X,Y) $ is Banach under the sup-norm
\[ \|f\|_\infty := \sup_{z \in B_X}\|f(z)\|. \]

Let $ (e_k)_{k \in \Gamma}$ be an orthonormal basis of $ X $ that we fix at once. Then every $ z \in X $ can be written as \[ z = \sum_{k \in \Gamma}z_ke_k, \quad  \overline{z} = \sum_{k \in \Gamma}\overline{z_k}e_k.\]

For $\varphi \in S(B_X),$  the set of  holomorphic self-maps on $ B_X$ we write $ \varphi(z) = \sum_{k\in \Gamma}\varphi_k(z) $ and $ \varphi'(z): X \to X$ its derivative at $ z, $ and $ R\varphi(z) = \langle \varphi'(z),\overline{z}\rangle $ its radial derivative at $ z. $


Given $ f \in H(B_X,Y) $ and $ z \in B_X. $  We define
\[\begin{aligned}
	\|\nabla f(z)\| &:= \sup_{u\in Y', \|u\|=1}\|\nabla(u\circ f)(z)\|, \\
	\|Rf(z)\|&:=  \sup_{u\in Y', \|u\|=1}|R(u\circ f)(z)|,
\end{aligned}\]
where
\[ R(u\circ f)(z) = \langle\nabla(u\circ f)(z), \overline{z}\rangle. \]
As above, it is obvious that $ \|Rf(z)\| \le \|\nabla f(z)\|\|z\| <  \|\nabla f(z)\| $ for every $ z \in B_X. $



\begin{defn}\label{weight}
	A positive continuous function $ \mu $ on the interval $ [0,1) $ is called normal  if there are  three constants $0 \le \delta < 1$ and $0 < a < b < \infty$ such that
	\begin{equation}\label{w_1}
		\frac{\mu(t)}{(1 - t)^{a}} \ \text{is decreasing on $[\delta, 1)$},\quad \lim_{t\to1}\frac{\mu(t)}{(1 - t)^{a}}=0, \tag{$ W_1 $}
	\end{equation}
	\begin{equation}\label{w_2}
		\frac{\mu(t)}{(1 - t)^{b}} \ \text{is increasing on $[\delta, 1)$},\quad \lim_{t\to1}\frac{\mu(t)}{(1 - t)^{b}} =\infty. \tag{$ W_2 $}
	\end{equation}
	If we say that a function $ \mu: B_X \to [0, \infty) $ is normal, we also assume that it is radial, that is, $ \mu(z) = \mu(\|z\|) $ for every $ z\in B_X. $
	
%
%
%
	
	Note that, since $ \mu $ is positive, continuous, $ m_\mu := \min_{t\in [0,\delta]}\mu(t) >0. $ Moreover, it follows from ($W_1$) that $ \mu  $ is increasing on $ [\delta,1),$ hence, we obtain that $  \max_{t \in [0,1)}\mu(t) =: M_\mu <\infty. $
	Then, it is easy to check that
	\begin{equation}\label{R_mu}
		\mu(z)\int_0^{\|z\|}\frac{dt}{\mu(t)} < R_\mu :=  \delta \frac{M_\mu}{m_\mu}+ 1-\delta <\infty \quad \text{for every}\ z \in B_X.
	\end{equation} 
\end{defn}
%

Throughout this paper, the weight $ \mu  $ always is assumed  to be   normal.  In the sequel,  when no confusion can arise,   we will use the symbol  $\lozenge$ to denote either $\nabla$ or $R.$


We define  Bloch-type spaces   on the unit ball $ B_X $ as follows:
\[
\mathcal B^\lozenge_\mu(B_X,Y) := \Big\{f \in H(B_X,Y):\ 
\|f\|_{s\mathcal B_\mu^\lozenge(B_X,Y)}:= \sup_{z \in B_X}\mu(z)\|\lozenge f(z)\| <\infty\Big\}. \]
It is easy to check $ \|\cdot\|_{s\mathcal B^\lozenge_\mu(B_X,Y)}  $  is a semi-norm on $  \mathcal B_\mu^\lozenge(B_X,Y) $  and this space is Banach under the sup-norm
\[
\|f\|_{\mathcal B_\mu^\lozenge(B_X,Y)} := \|f(0)\| + \|f\|_{s\mathcal B_\mu^\lozenge(B_X,Y)}.\]
We also define  little Bloch-type spaces   on the unit ball $ B_X $ as follows:
\[
\mathcal B^{\lozenge}_{\mu,0}(B_X,Y) := \Big\{f \in \mathcal B_\mu^\lozenge(B_X,Y):\ \lim_{\|z\|\to 1}\mu(z)\|\lozenge f(z)\| = 0\Big\} \]
endowed with the norm induced by $ \mathcal B_\mu^\lozenge(B_X,Y). $ 

In the case $ Y=\C $ we write $ \mathcal{B}_\mu^\lozenge(B_X),$  $ \mathcal{B}_{\mu,0}^\lozenge(B_X)$  instead of the respective notations.

It is clear that for every separating subspace $ W $ of $ Y $ we have
\[\mathcal B_\mu^\lozenge(B_X,Y) \subset WB_\mu^\lozenge(B_X)(Y),\quad   \mathcal B^{\lozenge}_{\mu,0}(B_X,Y) \subset W\mathcal B^{\lozenge}_{\mu,0}(B_X)(Y), \] 
%
%

For $ \mu(z) = 1 -\|z\|^2 $ we write $ \mathcal B^\lozenge(B_X, Y) $ instead of $ \mathcal B^\lozenge_\mu(B_X, Y) $ and when $ \dim X = m, $ $ Y =\C $ we obtain correspondingly the classical Bloch space $ \mathcal B^\lozenge(\BB_m). $
	
We will show below that the study of  Bloch-type spaces   on the unit ball can be reduced to studying functions defined on finite dimensional subspaces.

For each $ m \in \N $ we denote
\[ z_{[m]} := (z_1, \ldots,z_m) \in \BB_m\]
where $ \BB_m $ is the open unit ball in $ \C^m. $ For $ m \ge 2 $ by
\[ OS_m := \{x = (x_1, \ldots, x_m),\ x_k \in X, \langle x_k, x_j\rangle = \delta_{kj}\} \]
we denote the family of orthonormal systems of order $ m. $
 
It is clear that $ OS_1 $ is the unit sphere of $ X. $

For every $ x \in OS_m, $ $ f \in H(B_X,Y) $ we define
\[ f_x(z_{[m]})  = f\Big(\sum_{k=1}^mz_kx_k\Big).\]
Then  
\[ \nabla(u\circ f_x)(z_{[m]}) = \Big(\dfrac{\partial (u\circ f_x)}{\partial z_j}\Big(\sum_{k=1}^mz_kx_k\Big)\Big)_{j \in \Gamma}\quad \text{for every} \  u \in Y', \]
and hence
\begin{equation}\label{eq_2.1}
\Big\|\nabla f_x(z_{[m]})\Big\| =\Big\|\nabla f\Big(\sum_{k=1}^mz_kx_k\Big)\Big\|.
\end{equation} 
Now, for each finite subset $ F \subset \Gamma, $ in symbol $ |F| <\infty, $ we denote by $ \BB_{[F]} $ the unit ball of $ \text{span}\{e_k,\ k \in F\} $ and  $ f_F = f_x $ where $ x = \{e_k,\ k \in F\}.$  
For each $ z \in B_X $ and each $ F \subset \Gamma $ finite we write
\[ z_F = \sum_{k \in F}z_ke_k \in \BB_{[F]}. \]
\begin{defn} Let $ \BB_1 $ be the open unit ball in $ \C $ and $ f \in H(B_X,Y). $
	We define an affine semi-norm as follows
	\[ 	\|f\|_{s\mathcal{B}_\mu^{\rm aff}(B_X,Y)} := \sup_{\|x\|=1} \|f(\cdot\, x)\|_{s \mathcal B_\mu(\BB_1,Y)} 
	\]
	where $ f(\cdot\, x): \BB_1 \to Y $ given by $ f(\cdot\, x)(\lambda) = f(\lambda x)$ for every $ \lambda \in \BB_1, $ and
	\[ \|f(\cdot\, x)\|_{s \mathcal B_\mu^R(\BB_1,Y)} =\sup_{\lambda \in \BB_1}\mu(\lambda x)\|f'(\cdot\, x)(\lambda)\|. \]
	
	It is easy to see that  $\|\cdot\|_{s\mathcal{B}_\mu^{\rm aff}(B_X,Y)} $ is a semi-norm on $ \mathcal B_\mu(B_X,Y). $ We denote
	\[  \mathcal B_\mu^{\rm aff}(B_X,Y) := \{f \in \mathcal B_\mu (B_X,Y):\ \|f\|_{s\mathcal{B}_\mu^{\rm aff}(B_X,Y)}  <\infty\}. 
	\]
	It is also easy to check that   $ \mathcal B_\mu^{\rm aff}(B_X,Y) $ is Banach under the  norm
	\[\|f\|_{\mathcal B_\mu^{\rm aff}(B_X,Y)}:= \|f(0)\| + \|f\|_{s\mathcal B_\mu^{\rm aff}(B_X,Y)}.
	\]
	
	We also define  little affine Bloch-type spaces   on the unit ball $ B_X $ as follows:
	\[ \mathcal B_{\mu,0}^{\rm aff}(B_X,Y) := \big\{f \in \mathcal B_\mu^{\rm aff}(B_X,Y): \lim_{|\lambda| \to 1}\sup_{\|x\|=1}\mu(\lambda x)\|f'(\cdot\, x)(\lambda)\|=0\big\}. \]
	 As the  above, for $ \mu(z) = 1 - \|z\|^2 $ we use notation $ \mathcal B $ instead of $ \mathcal B_\mu. $
 \end{defn}
Now, let us recall a number of results proved in \cite{Qu}. 

\begin{prop}[\cite{Qu}, Proposition 4.1]\label{prop_OSm}
	Let $ f \in H(B_X,Y). $ The following are equivalent:
	\begin{enumerate}[\rm(1)]
		\item $ f \in \mathcal{B}_\mu^\nabla(B_X,Y); $
		\item $ \sup\{\|f_F\|_{\mathcal B_\mu^\nabla(\BB_{[F]},Y)}: \ F \subset \Gamma, |F| < \infty\}<\infty;$
\item $ \sup_{x \in OS_m}\|f_x\|_{\mathcal{B}_\mu^\nabla(\BB_m,Y)} < \infty  $ \ for every $ m \ge 2; $
\item There exists $ m \ge 2 $ such that $ \sup_{x \in OS_m}\|f_x\|_{\mathcal{B}_\mu^\nabla(\BB_m,Y)} < \infty.  $
	\end{enumerate}
Moreover, for each $ m\ge 2 $
\begin{equation}\label{eq_Prop4.1}
\|f\|_{s\mathcal{B}_\mu^\nabla(B_X,Y)} = \sup_{|F|<\infty}\|f_F\|_{s\mathcal{B}_\mu^\nabla(\BB_{[F]},Y)}= \sup_{x\in OS_m}\|f_x\|_{s\mathcal{B}_\mu^\nabla(\BB_m,Y)}. 
\end{equation}  
\end{prop}

\begin{rmk} The proposition is not true for the case $ m=1. $ 
%
\end{rmk}

\begin{prop}[\cite{Qu}, Proposition 4.2]\label{prop_OSm_0}
	Let $ f \in H(B_X,Y). $ The following are equivalent:
	\begin{enumerate}[\rm(1)]
		\item $ f \in \mathcal{B}_{\mu,0}^\nabla(B_X,Y); $
		\item $\forall \varepsilon>0\ \exists \varrho>0\ \forall z \in B_X\ \text{with}\ \|z_{[F]}\|>\varrho$ for every $ F\subset \Gamma, |F|<\infty $
		$$ \sup_{F\subset\Gamma, |F|<\infty}\mu(z_{[F]})\|\nabla f_F(z_{[F]})\|<\varepsilon;$$
		\item  $\forall \varepsilon>0\ \exists \varrho>0\ \forall z \in B_X\ \text{with}\ \|z_{[m]}\|>\varrho$ for every $ m\ge2 $
		$$\sup_{m\ge2} \sup_{x \in OS_m}\mu(z_{[m]})\|\nabla f_x(z_{[m]})\|<\varepsilon; $$
		\item  $\exists  m \ge 2 \ \forall \varepsilon>0\ \exists \varrho>0\ \forall z \in B_X\ \text{with}\ \|z_{[m]}\|>\varrho$
		$$\sup_{x \in OS_m}\mu(z_{[m]})\|\nabla f_x(z_{[m]})\|<\varepsilon.$$
	\end{enumerate}
\end{prop}

\begin{thm}[\cite{Qu}, Theorem 4.7]\label{thm_main}
	The spaces $ \mathcal{B}_\mu^\nabla(B_X,Y), $ $ \mathcal{B}_\mu^R(B_X,Y) $ and $ \mathcal{B}_\mu^{\rm aff}(B_X,Y)$ coincide. The spaces $ \mathcal{B}_{\mu,0}^\nabla(B_X,Y), $ $ \mathcal{B}_{\mu,0}^R(B_X,Y) $ and $ \mathcal{B}_{\mu,0}^{\rm aff}(B_X,Y)$ coincide. Moreover,
	\[ \|f\|_{\mathcal B_\mu^R(B_X,Y)} \le \|f\|_{\mathcal B_\mu^\nabla(B_X,Y)} \le 2\sqrt{2}R_\mu\|f\|_{\mathcal B_\mu^{\rm aff}(B_X,Y)}.  \]
\end{thm}
Next, we present a M\"obius invariant norm for the  Bloch-type space $W\mathcal{B}(B_X,Y). $

M\"obius transformations on a Hilbert space $ X $ are the mappings $ \varphi_a,  $ $ a \in B_X, $  defined as follows:
\begin{equation}\label{eq_Mobius}
 \varphi_a(z) = \frac{a - P_a(z) - s_aQ_a(z)}{1 - \langle z, a\rangle}, \quad z \in \B_X
\end{equation} 
where $s_a = \sqrt{1 - \|a\|^2},$   $P_a$ is the orthogonal projection from $X$ onto the one dimensional subspace $[a]$ generated by a, and $Q_a$ is the orthogonal projection from $X$ onto $X \ominus [a].$ 

When $a = 0,$ we simply define $\varphi _a(z)=-z.$ 
It is obvious that each $\varphi_a$  is a holomorphic mapping from $B_X$ into $X.$

\begin{defn}
Let $ X $ be a complex Hilbert space, $ Y $ be a Banach space and $ f \in H(B_X,Y). $ The invariant gradient norm
\[ \|\widetilde{\nabla}f(z)\| := \|\nabla(f \circ \varphi_z)(0)\|\quad\text{for any $z \in B_X. $}  \]

	We define invariant semi-norm   as follows
	\[ 		\|f\|_{s \mathcal B^{\rm inv}(B_X,Y)}  := \sup_{z\in B_X}\|\widetilde\nabla f(z)\| =\sup_{z\in B_X}\sup_{u\in Y', \|u\|\le1}\|\widetilde\nabla (u\circ f)(z)\|. \]
  We denote
	\[  		\mathcal B^{\rm inv}(B_X,Y) := \{f \in \mathcal B (B_X,Y):\ \|f\|_{s \mathcal B^{\rm inv}(B_X,Y)}  <\infty\}. \]
	
	It is also easy to check that $  \mathcal B^{\rm inv}(B_X,Y) $   is Banach under the  norm 
	\[ 		\|f\|_{\mathcal B^{\rm inv}(B_X,Y)}  := \|f(0)\| + \|f\|_{s\mathcal B^{\rm inv}(B_X,Y)}. \] 
\end{defn}

\begin{thm}[\cite{Qu}, Theorem 4.9]\label{thm_main1}
	The spaces $ \mathcal{B}^\nabla(B_X,Y), $ and $ \mathcal{B}^{\rm inv}(B_X,Y) 
	$  coincide. Moreover,
	\[ \|f\|_{\mathcal B^\nabla(B_X,Y)} \le \|f\|_{\mathcal B^{\rm inv}(B_X,Y)} \lesssim \|f\|_{\mathcal B^\nabla(B_X,Y)}.  \]
\end{thm}
Now let  $ W\subset Y' $ be a separating subspace of the dual $ Y', $ 
 then we obtain the equivalence of the norms in associated Bloch-type spaces:
\[ \|\cdot\|_{W\mathcal B_\mu^R(B_X)} \cong \|\cdot\|_{W\mathcal B_\mu^\nabla(B_X)} \cong \|\cdot\|_{W\mathcal B_\mu^{\rm aff}(B_X)},\]
\[ \|\cdot\|_{W\mathcal B^R(B_X)} \cong \|\cdot\|_{W\mathcal B^\nabla(B_X)} \cong \|\cdot\|_{W\mathcal B^{\rm aff}(B_X)} \cong \|\cdot\|_{W\mathcal B^{\rm inv}(B_X)}.\]


Hence, for the sake of simplicity, from now on we write $ \mathcal B_\mu $ instead of $ \mathcal B_\mu^R. $

\medspace
\medspace
Finally, the following  shows that   $ W\mathcal{B}_\mu(B_X)(Y), $ $ W\mathcal{B}_{\mu,0}(B_X)(Y) $ 
satisfy   (we1)-(we3).

%
%
\begin{prop}[\cite{Qu}, Proposition 4.11]\label{prop_e1e3}Let
	$W \subset Y'$ be a separating subspace. Let $ \mu $ be a   normal weight on $ B_X. $  Then $  \mathcal{B}_\mu(B_X), $ $  \mathcal{B}_{\mu,0}(B_X) $  satisfy (e1) and (e2), and hence,  $ W\mathcal{B}_\mu(B_X)(Y)$ and  $ W\mathcal{B}_{\mu,0}(B_X)(Y) $ satisfy  (we1)-(we3). \end{prop}

\begin{rmk}\label{rmk_1}
	(1) In the proof of this proposition we used the following estimate
	\begin{equation}\label{eq_5.1}
		|f(z)| \le \max\bigg\{1,\int_0^{\|z\|}\frac{dt}{\mu(t)}\bigg\}\|f\|_{\mathcal{B}_\mu(B_X)}\quad \forall f \in \mathcal{B}_\mu(B_X), \forall z \in B_X.
	\end{equation} 
	In fact, the estimate  (\ref{eq_5.1}) can be written as follows
	\[ |f(z)| \le |f(0)| + \int_0^{\|z\|}\frac{dt}{\mu(t)}\|f\|_{s\mathcal B_\mu}.\]
	
	(2) Since $ \mu $ is non-increasing, it is easy to see that  $ M_\mu := \sup_{t \in [0,1)}\mu(t) <\infty. $ Then, from (\ref{eq_5.1}) we obtain that
	\begin{equation}\label{eq_5.1A}
		\mu(z)|f(z)| \le \max\{1, M_\mu\} \|f\|_{\mathcal B_{\mu}(B_X)}\quad \forall z \in B_X.
	\end{equation}
%
%
\end{rmk}
%
%
%

\section{The Main Results and Test Functions}
\setcounter{equation}{0}
This section  introduces the mains results of paper and  test functions that are  usefull for the proofs.

In this section we consider  $ \nu $ is a non-increasing, normal weight on $ B_X $ and $\varphi   \in S(B_X).$

We begin this section by  constructing  test functions that are  usefull for the proofs of our main results.

\subsection{The main results}
Let $ \nu, \mu $ be  normal weights on $ B_X. $
Let $\varphi \in S(B_X),$  the set of  holomorphic self-maps on $ B_X$ and $ \psi \in H(B_X). $ For each   $ F \subset \Gamma $  finite and $ m \in \N $ we write
\[ \varphi^{[F]} = \varphi\big|_{\text{span}\{e_k, k \in F\}}, \quad \varphi^{[m]} = \varphi\big|_{\text{span}\{e_1, \ldots, e_m\}}\]
and
\[ \psi^{[F]} = \psi\big|_{\text{span}\{e_k, k \in F\}}, \quad \psi^{[m]} = \psi\big|_{\text{span}\{e_1, \ldots, e_m\}}.\]
For each $ j \in \Gamma $ and $ k\ge 1 $ we denote
\[ \varphi_j(\cdot) := \langle\varphi(\cdot), e_j\rangle, \quad \varphi_{(k)}(\cdot) := (\varphi_1(\cdot), \ldots, \varphi_k(\cdot)).  \]
In this section we investigate the boundedness and the compactness of  the   operators $ W_{\psi,\varphi} $ between the (little) Bloch-type spaces $ \mathcal B_\nu $ ($ \mathcal B_{\nu,0} $)  and $ \mathcal B_\mu,$    ($ \mathcal B_{\mu,0}$) via the estimates of $ \psi^{[m]}, \varphi^{[m]} $ and $ \varphi_{(k)}. $ 
Hence, by Theorem \ref{thm_WO}, some characterizations for the boundedness and compactness of the operators $ \widetilde{C}_{\psi, \varphi} $ between  spaces $ W\mathcal B_\nu(B_X,Y), $ ($ W\mathcal B_{\nu,0}(B_X,Y) $) and  $ W\mathcal B_{\mu}(B_X,Y), $ ($ W\mathcal B_{\nu,0}(B_X,Y) $)  will be obtained from these results.

By  Theorem \ref{thm_main} in the paper we will only present the results for the spaces $ \mathcal B_\mu^R(B_X). $

First we investigate the boundedness of  extended Ces\`aro composition operators.  We use there certain quantities, which will be used in this work. We list them below:

\begin{subequations}
	\begin{align}
		&\mu^{[m]}(y)\|R\psi^{[m]}(y)\|\max\bigg\{1, \int_0^{\|\varphi^{[m]}(y)\|}\dfrac{dt}{\nu(t)}\bigg\}, \label{B_1}\\
		&\mu^{[F]}(y)\|R\psi^{[F]}(y)\| \max\bigg\{1, \int_0^{\|\varphi^{[F]}(y)\|}\dfrac{dt}{\nu(t)}\bigg\}, \label{B_2} \\
		&\mu(y)\|R\psi(y)\|\max\bigg\{1, \int_0^{\|\varphi_{(k)}(y)\|}\dfrac{dt}{\nu(t)}\bigg\}, \label{B_3} \\
		&\mu(y)\|R\psi(y)\| \max\bigg\{1, \int_0^{\|\varphi(y)\|}\dfrac{dt}{\nu(t)}\bigg\}. \label{B_4}
	\end{align}
\end{subequations}
We also use the notations $ C_{\psi, \varphi}: \mathcal B_\nu  \to \mathcal B_\mu,$  $C^0_{\psi, \varphi}: \mathcal B_{\nu,0} \to \mathcal B_{\mu},$ $C^{0,0}_{\psi, \varphi}: \mathcal B_{\nu,0} \to \mathcal B_{\mu,0}$ to denote the  extended Ces\`aro composition  operators.

Now we are ready to state the first main result. The first ones  propose to the boundeness of $ C_{\psi,\varphi}, $ $ C^0_{\psi,\varphi}, $ $ C^{0,0}_{\psi,\varphi} $ and  the equivalent relationships between them.

\begin{thm}
	\label{thm6_1} Let $\psi \in H(B_X),$  $\varphi   \in S(B_X)$  and  $ \mu, \nu $ be   normal weights on $ B_X. $  Then   the  following are equivalent:  
	\begin{enumerate}[\rm(1)]
		\item $ M^{[m]}_{R\psi,\varphi} := \displaystyle\sup_{y \in \BB_m}(\ref{B_1})<\infty $ for some  $ m \ge 2;$ 
		
		\item $  M^{[F]}_{R\psi,\varphi} := \displaystyle\sup_{y \in \BB_{[F]}}(\ref{B_2})<\infty $  for every $ F \subset \Gamma$ finite;
		
		\item $ M^{(k)}_{R\psi,\varphi} := \displaystyle\sup_{y \in B_X}(\ref{B_3})<\infty $ for every $ k \ge 1;$    
	
		\item $ M_{R\psi,\varphi}:= \displaystyle\sup_{y \in B_X}(\ref{B_4}) < \infty;$ 
				\item $C_{\psi, \varphi}:   \mathcal B_\nu(B_X) \to \mathcal B_\mu(B_X)$     is bounded; 
		\item $C^0_{\psi, \varphi}: \mathcal B_{\nu,0}(B_X) \to \mathcal B_\mu(B_X)$     is bounded. 
	\end{enumerate}
	Moreover, if  $C_{\psi, \varphi}$     is bounded,  the following asymptotic relation holds for some $ m\ge 2: $
	\begin{equation}\label{eq_Thm6.1}
		\|C_{\psi,\varphi}\|    \asymp  M^{[m]}_{R\psi,\varphi}.
	\end{equation}
\end{thm}
Next, we will touch the characterizations for  the boundedness of the operators from $ \mathcal B_{\nu,0}(B_X)$  to  $\mathcal B_{\mu,0}(B_X).$

\begin{thm}\label{thm2_2} Let $\psi \in H(B_X),$  $\varphi   \in S(B_X)$ and  $ \mu, \nu $ be   normal weights on $ B_X. $  Then   the  following are equivalent:
	\begin{enumerate}[\rm(1)]
		\item   $\psi^{[m]} \in \mathcal B_{\mu^{[m]},0}(\BB_m)$  and  $ M^{[m]}_{R\psi,\varphi} := \displaystyle\sup_{y \in \BB_m}(\ref{B_1})<\infty $ for some  $ m \ge 2;$ 
		\item   $\psi^{[F]} \in \mathcal B_{\mu^{[F]},0}(\BB_{[F]})$  and   $  M^{[F]}_{R\psi,\varphi} := \displaystyle\sup_{y \in \BB_{[F]}}(\ref{B_2})<\infty $  for every $ F \subset \Gamma$ finite;
		\item  $\psi \in \mathcal B_{\mu,0}(B_X)$  and $ M^{(k)}_{R\psi,\varphi} := \displaystyle\sup_{y \in B_X}(\ref{B_3})<\infty $ for every $ k \ge 1;$    
		\item  $\psi  \in \mathcal B_{\mu,0}(B_X)$ and $ M_{R\psi,\varphi}:= \displaystyle\sup_{y \in B_X}(\ref{B_4}) < \infty;$ 
		\item  $C^{0,0}_{\psi, \varphi}: \mathcal B_{\nu,0}(B_X) \to \mathcal B_{\mu,0}(B_X)$     is bounded. 	
	\end{enumerate}
	In this case, the following  asymptotic relation    holds  for some $ m\ge 2: $
	\begin{equation} \label{eq_Thm6.2}
		\|C^{0,0}_{\psi,\varphi}\|    \asymp  M^{[m]}_{R\psi,\varphi}.
	\end{equation}
\end{thm}
Now the following are characterizations of  the compactness of extended Ces\`aro composition operators  $ C_{\psi,\varphi}. $
\begin{thm}\label{thm_compact} Let $\psi \in H(B_X),$  $\varphi \in S(B_X)$ and  $ \mu, \nu $ be   normal weights on $ B_X$ such that $ \int_0^1\frac{dt}{\nu(t)} = \infty. $ 
	\begin{enumerate}[\rm (A)]
		\item The  following are equivalent:
		\begin{enumerate}[\rm (1)]
			\item $ (\ref{B_3})\to0 $ as $ \|\varphi_{(k)}(y)\| \to 1 $	 or every $ k\ge1; $
			\item $C_{\psi, \varphi}: \mathcal B_{\nu}(B_X) \to \mathcal B_{\mu}(B_X)$     is compact; 
			\item $C^0_{\psi, \varphi}: \mathcal B_{\nu, 0}(B_X) \to \mathcal B_{\mu}(B_X)$     is compact.
		\end{enumerate}
		\item Under the additional assumption   that  there exists $ m\ge 2 $ such that     
		\begin{equation}\label{phi_relative_compact} 
			\begin{aligned}
				B[\varphi^{[m]}, r] &:=\big\{\varphi^{[m]}(y): \|\varphi^{[m]}(y)\| < r,  y \in \BB_m \big\}\\
				&\quad\ \ \text{is relatively compact for every $ 0\le r<1,$ }
			\end{aligned} 
		\end{equation}
		the assertions (2), (3) and  following are equivalent:
		\begin{enumerate}
			\item[\rm(4)]  $ (\ref{B_1})\to0 $ as $ \|\varphi^{[m]}(y)\| \to 1; $ 
			\item[\rm(5)]  $ (\ref{B_2})\to0 $ as $ \|\varphi^{[F]}(y)\| \to 1 $  for every $ F \subset \Gamma$ finite;
			\item[\rm(6)] $ (\ref{B_4})\to0 $ as $ \|\varphi(y)\| \to 1. $	 
		\end{enumerate} 
	\end{enumerate} 
\end{thm}
 \begin{rmk}\label{rmk_6.1}  \begin{itemize}
		\item[\rm(a)]  The implication (3) $ \Rightarrow $ (4) was proved without appeal to the assumption (\ref{phi_relative_compact}). 
		\item[\rm(b)] In general, the assumption (2) or (3)    does not imply   (\ref{phi_relative_compact}). This means that if (\ref{phi_relative_compact}) was moved into the hypothesis of (B) then the equivalence of the statements in (B) of Theorem \ref{thm_compact}   is broken.
		An illustrative example of this comment will be introduced after the following theorem:
	\end{itemize}
\end{rmk}
 \begin{thm}\label{thm_6.4}
	Let $\psi \in H(B_X)$  and $\varphi \in S(B_X).$ Assume that $ C_{\psi, \varphi}: \mathcal B_\nu(B_X) \to \mathcal B_\mu(B_X) $ is compact. Then
	\begin{equation}\label{compact_3}   \varphi(r B_X)  \ \text{is relatively compact for every $ 0\le r < 1$}.
	\end{equation}
\end{thm}
We introduce  below is an example which shows that (\ref{compact_3}) does not imply  (\ref{phi_relative_compact}).

\begin{exam} Let $ \{e_j\}_{j\ge 1} $ be an orthonormal sequence in a Hilbert space $ X. $   Consider the function $ \varphi \in S(B_X) $ given by
	\[ \varphi(z) := \sum_{n=1 }^\infty\langle z,e_n\rangle^ne_n  \quad \forall z  \in B_X. \]
	It is easy to check that $ \varphi(rB_X) $ is relatively compact for every $ 0<r<1. $ 
	
	Now we show that  $ B[\varphi,\frac{1}{2}]  $ is not relatively compact.
	Consider the sequence $ \{z_k\}_{k\ge1} \subset B_X$ given by
	\[ z_k =  \frac1{\sqrt[k]{4}}e_k\quad \forall k \ge1. \]
	It is obvious $ \|\varphi(z_k)\|<\frac12 $ for every $ k\ge1. $ Then for every $ k\ge 1 $ and $ s>1 $ we have
	\[ \|\varphi(z_k) - \varphi(z_{k+s})\| =  \frac{\sqrt{2}}{4}.\]
	Thus, we get the desired claim.
\end{exam}

%
%
%

\begin{cor}\label{cor_6.7}
	Assume that $\sup_{z\in X}\|\varphi(z)\| <\infty.$  Then  $ C_{\psi,\varphi}, $ $C^0_{\psi,\varphi} $ are compact if and only if $ \varphi(B_X) $ is relatively compact.
\end{cor}
Indeed, in this case, by the hypotheses,  $ (\ref{B_4})\to0 $ as $ \|\varphi(y)\| \to 1. $	 and  (\ref{phi_relative_compact}),  (\ref{compact_3})   always hold.

\begin{thm}\label{thm_compact2} 
	Let $\psi \in H(B_X)$  and $\varphi \in S(B_X)$    such that      (\ref{phi_relative_compact}) holds for some $ m\ge 2. $
	Let $ \mu, \nu $ be  normal weights on $ B_X $ such that $ \int_0^1\frac{dt}{\nu(t)} < \infty. $  Then   the  following are equivalent:
	\begin{enumerate}[\rm(1)]
		\item  $ \psi^{[m]} \in \mathcal B_{\mu^{[m]}}(\BB_m); $ 
		\item  $ \psi^{[F]} \in \mathcal B_{\mu^{[F]}}(\BB_{[F]})$ for every $ F \subset \Gamma$ finite;
		\item  $ \psi \in \mathcal B_\mu(B_X); $ 
		\item $C_{\psi, \varphi}: \mathcal B_{\nu}(B_X) \to \mathcal B_{\mu}(B_X)$     is compact; 
		\item $C^0_{\psi, \varphi}: \mathcal B_{\nu, 0}(B_X) \to \mathcal B_{\mu}(B_X)$     is compact.
	\end{enumerate} 
\end{thm}
Next, we discus the  compactness of the operator $C^{0,0}_{\psi, \varphi}: \mathcal B_{\nu,0}(B_X) \to \mathcal B_{\mu,0}(B_X).$ 

\begin{thm}\label{thm2_2_cp} Let $\psi \in H(B_X),$ $\varphi   \in S(B_X)$ and  $ \mu, \nu $ be  normal weights on $ B_X. $  
	Then   
	\begin{enumerate}
		\item [\rm(A)] The following are equivalent:
		\begin{enumerate}
			\item [\rm(1)]  $ (\ref{B_3})\to 0 $ as $ \|\varphi_{(k)}\|\to 1 $ for every $ k\ge 1; $
			\item[\rm(2)] 	$C^{0,0}_{\psi, \varphi}: \mathcal B_{\nu,0}(B_X) \to \mathcal B_{\mu,0}(B_X)$     is compact,
		\end{enumerate}
		\item[\rm(B)] Under the additional assumption   that  there exists $ m\ge 2 $ such that     
		(\ref{phi_relative_compact}) holds, the assertions (2) and  following are equivalent:
		\begin{enumerate}
			\item[\rm(3)]  $ (\ref{B_1})\to0 $ as $ \|\varphi^{[m]}(y)\| \to 1; $	 
			\item[\rm(4)]  $ (\ref{B_2})\to0 $ as $ \|\varphi^{[F]}(y)\| \to 1 $  for every $ F \subset \Gamma$ finite;
			\item[\rm(5)] $ (\ref{B_4})\to0 $ as $ \|\varphi(y)\| \to 1. $	 
		\end{enumerate} 
	\end{enumerate}
\end{thm}
\begin{cor}\label{cor2_2_cp} Let $\psi \in H(B_X),$  $\varphi   \in S(B_X)$   and  $ \mu, \nu $ be   normal weights on $ B_X $  such that    $ \int_0^1\frac{dt}{\nu(t)} = \infty. $ 
	Then 
	\begin{enumerate}
		\item [\rm (A)]  The  following are equivalent:
		\begin{enumerate}
			\item [\rm(1)]  $ (\ref{B_3})\to 0 $ as $ \|\varphi_{(k)}(y)\|\to 1 $ for every $ k\ge 1; $ 
			\item [\rm(2)] $C^{0}_{\psi, \varphi}: \mathcal B_{\nu,0}(B_X) \to \mathcal B_{\mu}(B_X)$     is compact;
			\item [\rm(3)] $C^{0,0}_{\psi, \varphi}: \mathcal B_{\nu,0}(B_X) \to \mathcal B_{\mu,0}(B_X)$     is compact;
			\item[\rm(4)] $C_{\psi, \varphi}: \mathcal B_{\nu}(B_X) \to \mathcal B_{\mu}(B_X)$     is compact,
		\end{enumerate}
		\item[\rm(B)] Under the additional assumption   that  there exists $ m\ge 2 $ such that     
		(\ref{phi_relative_compact}) holds, the assertions (2), (3), (4) and  following are equivalent:
		\begin{enumerate}
			\item[\rm(5)]  $ (\ref{B_1})\to0 $ as $ \|\varphi^{[m]}(y)\| \to 1; $	 
			\item[\rm(6)]  $ (\ref{B_2})\to0 $ as $ \|\varphi^{[F]}(y)\| \to 1 $  for every $ F \subset \Gamma$ finite;
			\item[\rm(7)] $ (\ref{B_4})\to0 $ as $ \|\varphi(y)\| \to 1. $	 
		\end{enumerate}	
	\end{enumerate}
\end{cor}
\begin{cor}\label{thm2_2_cp1} Let $\psi \in H(B_X),$  $\varphi   \in S(B_X)$ and   $ \mu, \nu $ be   normal weights on $ B_X $  such that    $ \int_0^1\frac{dt}{\nu(t)} < \infty$  and (\ref{phi_relative_compact}) holds  for every $ m\ge 1. $
	Then, the  following are equivalent: 
	\begin{enumerate}[\rm(1)]
		\item  $ \psi^{[m]} \in \mathcal B_{\mu^{[m]},0}(\BB_m); $ 
		\item  $ \psi^{[F]} \in \mathcal B_{\mu^{[F]},0}(\BB_{[F]})$ for every $ F \subset \Gamma$ finite;
		\item  $ \psi \in \mathcal B_{\mu,0}(B_X); $ 
		\item $C^{0,0}_{\psi, \varphi}: \mathcal B_{\nu, 0}(B_X) \to \mathcal B_{\mu,0}(B_X)$     is compact.
	\end{enumerate}
\end{cor}
\begin{rmk} Under the additional condition that $ \varphi(0) = 0,$ the limits in Theorems \ref{thm_compact}, \ref{thm2_2_cp},  Corollary \ref{cor2_2_cp}     are replaced by similar ones but with $ \|z_{k}\| \to 1, $ $ \|z_{[m]}\| \to 1,$ $ \|z_{[F]}\ \to 1,$  $ \|z\| \to 1 $ respectively.
	Indeed, in a more general framework, it suffices to  show $ \|\varphi(z)\|\le \|z\| $ for every $ z\in B_X. $ That means we  have to give  an infinite version of Schwarz's lemma. 
	
	For each $ z \in X, z \neq 0 $ and $ w \in \overline{B_X}, $ applying classical Schwarz's lemma to the functions $ \phi_{z,w}: \BB_1 \to \BB_1$ given by
	\[ \phi_{z,w}(t) := \langle\varphi(tz/\|z\|), w\rangle \quad\forall t \in \BB_1, \]
	we have 
	\[ |\phi_{z,w}(t)| \le |t|. \]
	Then, choosing $ t = \|z\| $ and $ w = \frac{\overline{\varphi(z)}}{\|\varphi(z)\|} $ we get the desired inequality.
\end{rmk}

To finish this paper, combining Theorem \ref{thm_WO} with the main results in Section 6, we state the characterizations for the boundedness and the compactness of the operators   $ \widetilde{C}_{\psi,\varphi},  $  $ \widetilde{C}^0_{\psi,\varphi},  $ $ \widetilde{C}^{0,0}_{\psi,\varphi}. $

\begin{thm}
	Let   $W \subset Y'$ be a  separating  subspace. Let $\psi \in H(B_X),$  $\varphi   \in S(B_X)$  and $ \mu, \nu $ be   normal weights on $ B_X. $ Then 
	\begin{enumerate}[\rm (1)]
		\item The  following are equivalent:
		\begin{itemize}
			\item $\widetilde{C}_{\psi,\varphi} $ is bounded;
			\item $\widetilde{C}^0_{\psi,\varphi} $ is bounded;
			\item  One of the assertions (1)-(4) in Theorem \ref{thm6_1},
		\end{itemize}
		\item 
		The  following are equivalent:
		\begin{itemize}
			\item $\widetilde{C}^{0,0}_{\psi,\varphi}$ is bounded;
			\item  One of the assertions (1)-(4) in Theorem \ref{thm2_2}.
		\end{itemize} 
	\end{enumerate}
\end{thm}

\begin{thm}
	Let   $W \subset Y'$ be a  almost norming subspace. Let $\psi \in H(B_X),$ $\varphi   \in S(B_X)$ and $ \mu, \nu $ be   normal weights on $ B_X. $ Assume that one of the following is satisfied:
	\begin{enumerate}[\rm(a)]
		\item $ \int_0^1\frac{dt}{\nu(t)} = \infty $ and  the assertion (1) in Theorem \ref{thm_compact};
		\item $ \int_0^1\frac{dt}{\nu(t)} = \infty $ and  one of the assertions (4)-(6) in Theorem \ref{thm_compact}  with the additional assumption that (\ref{phi_relative_compact}) holds for some $ m\ge2; $
		\item $ \int_0^1\frac{dt}{\nu(t)} < \infty $ and  one of the assertions (1)-(3) in Theorem \ref{thm_compact2} holds  under the additional assumption that (\ref{phi_relative_compact}) holds for some $ m\ge2.$
	\end{enumerate}
	Then  the  following are equivalent:
	\begin{enumerate}[\rm (1)]
		\item $\widetilde{C}_{\psi,\varphi} $ is (resp. weakly) compact;
		\item $\widetilde{C}^0_{\psi,\varphi} $ is (resp. weakly) compact;
		\item The identity map $ I_Y: Y \to Y $ is   (resp. weakly) compact.
	\end{enumerate}
\end{thm}
\begin{thm}
	Let   $W \subset Y'$ be a  almost norming subspace. Let $\psi \in H(B_X),$ $\varphi   \in S(B_X)$ 
	and  $ \mu, \nu $ be   normal weights on $ B_X. $ Assume that one of the following is satisfied:
	\begin{enumerate}[\rm(a)]
		\item The assertion (1) in Theorem \ref{thm2_2_cp};
		\item One of   the assertions (3)-(5) in Theorem \ref{thm2_2_cp} with the additional assumption that (\ref{phi_relative_compact}) holds for some $ m\ge2.$
	\end{enumerate}
	Then  the  following are equivalent:
	\begin{enumerate}[\rm (1)]
		\item   $\widetilde{C}^{0,0}_{\psi,\varphi} $ is (resp. weakly) compact,  
		\item  The identity map $ I_Y: Y \to Y $ is (resp. weakly) compact.
	\end{enumerate}
\end{thm}

\subsection{The test functions}

%

Let $ \nu $ be a normal weight on $ B_X. $ First we consider the holomorphic function 
\begin{equation}\label{test_g} g(z) := 1 + \sum_{k > k_0}2^kz^{n_k} \quad \forall z \in \BB_1
\end{equation}   
where $ k_0 = \big[\log_2\frac{1}{\nu(\delta)}\big], $ $ n_k = \big[\frac{1}{1-r_k}\big] $ with $ r_k = \nu^{-1}(1/2^k) $ for every $ k\ge 1, $ and $ \delta $ is the constant in Definition \ref{weight}. Here the symbol $ [x] $ means the greatest integer not more than $ x. $ 
By \cite[Theorem 2.3]{HW}, $ g(t) $ is increasing on $ [0,1) $ and
\begin{equation}\label{test_g1}
	|g(z)| \le g(|z|) \in \R\quad \forall z \in \BB_1,
\end{equation}
\begin{equation}\label{test_g2}
	0< C_1 :=\inf_{t\in [0,1)}\nu(t)g(t) \le \sup_{t\in [0,1)}\nu(t)g(t)  \le C_2 := \sup_{z \in \BB_1}\nu(z)|g(z)|<\infty. 
\end{equation}
Moreover, there exists a positive constant $ C_3 $ such that the inequality
\begin{equation}\label{test_g2a}
	\int_0^r g(t)dt \le C_3 \int_0^{r^2}g(t)dt
\end{equation}
holds for all $ r \in [r_1, 1), $ where $ r_1 \in (0, 1) $ is a constant such that
$  \int_0^{r_1}g(t)dt = 1$ (see \cite{Qu}, Proposition 5.1).

Now, for $ w \in B_X $ fixed consider  put the test functions
\begin{equation}\label{test_func2}
	\beta_{w}(z) := \int_0^{\langle z,w\rangle}g(t)dt,\quad z\in B_X,
\end{equation} 
 and for  $ \|w\| > r $ for some $ r>0, $ put
\begin{equation}\label{test_func3}
	\gamma_{w}(z) := \dfrac{1}{\int_0^{\|w\|^2}g(t)dt}\Bigg(\int_0^{\langle z,w\rangle}g(t)dt\Bigg)^2,\quad z\in B_X,
\end{equation}

\begin{prop}[\cite{Qu}, Propositions 5.2 and 5.3]\label{lem2} 
	We have 
		\begin{enumerate}[\rm (1)]
			\item $ \beta_{w}, \gamma_{w} \in \mathcal B_{\nu,0}(B_X) $ and 
			$$ \|\beta_{w}\|_{\mathcal B_\nu(B_X)} \le C_2, \, \|\gamma_{w}\|_{\mathcal B_\nu(B_X)} \le 2C_2. $$
		\item  The sequences $ \{\beta_{w^n}\}_{n\ge 1}, $  $ \{\gamma_{w^n}\}_{n\ge1} $   are bounded in $ \mathcal B_\nu(B_X); $ 
		\item $ 	\gamma_{w^n} \to 0 $  
		uniformly
		on any compact subset of $ B_X$  if $ \int_0^1\frac{dt}{\nu(t)} =\infty.$
	\end{enumerate}
		 \end{prop}

We also  need the following lemma:
\begin{lem}\label{nabla} For every $f, \psi \in H(B_X),$   and  $\varphi \in S(B_X)$ we have
	\begin{equation}\label{nabla_1} R(C_{\psi,\varphi}f)(z) = f(\varphi(z))R\psi(z) \quad \forall z \in B_X.
	\end{equation}
\end{lem}
Indeed, assume that $ \sum_{n=1}^\infty P_n(z) $ is the Taylor series of $ (f\circ \varphi)(z)R\psi(z), $ where $ P_n $ is a homogeneous polynomial of degree $ n. $ Then we have
$$ R(C_{\psi,\varphi})(z) = R\int_0^1\sum_{n=1}^\infty P_n(z)t^n\dfrac{dt}{t} = R\sum_{n=1}^\infty\dfrac{P_n(z)}{n} = \sum_{n=1}^\infty P_n(z) = f( \varphi(z))R\psi(z).$$

\section{Proof of Theorems \ref{thm6_1} and \ref{thm2_2}}
\setcounter{equation}{0}

\subsection{Proof of Theorem \ref{thm6_1}}
\begin{proof}[Proof of Theorem \ref{thm6_1}]
It is clear  that  (4) $ \Rightarrow$ (2) $ \Rightarrow $ (1) and (4) $ \Rightarrow$ (3) and, since $ \mathcal B_{\nu,0}(B_X) \subset \mathcal B_\nu(B_X),$    the implication (5) $\Rightarrow$ (6)   is obvious.

\medskip
(1) $\Rightarrow$ (5): Assume that $M^{[m]}_{R\psi, \varphi} < \infty$   for some $ m\ge 2. $ For each $ x \in OS_m $ we write $ z_x := \sum_{k=1}^mz_kx_k. $ Note that $ \|z_x\| = \|z_{[m]}\| $ and hence $ \mu^{[m]}(z_{[m]}) = \mu^{[m]}(z_x). $
By (\ref{eq_5.1}) 
and
 (\ref{nabla_1}), for every $ f \in \mathcal{B}_\nu(B_X) $ 
 and for every $ x \in OS_m $
  we have 
$$
\aligned
&\|(C_{\psi,\varphi}(f))_x\|_{s\mathcal{B}_{\mu^{[m]}}(\BB_m)} =    \sup_{z_x \in \BB_m}\mu^{[m]}(z_x)\|R(C_{\psi^{[m]},\varphi^{[m]}}f_x(z_{[m]})))\| \\
&= \sup_{z_x \in \BB_m}\mu^{[m]}(z_x)|(f\circ\varphi)_x(z_{[m]}))| \|R\psi^{[m]}(z_x)\| \\
&= \sup_{z_x \in \BB_m}\mu^{[m]}(z_x)|f(\varphi^{[m]}(z_x))| \|R\psi^{[m]}(z_x)|\| \\
&\le \sup_{z_x \in \BB_m} \mu^{[m]}(z_x) \|R\psi^{[m]}(z_x)\|  \max\bigg\{1, \int_0^{\|\varphi^{[m]}(z_x)\|}\dfrac{dt}{\nu(t)}\bigg\}\|f\|_{\mathcal B_\nu(B_X)} \\
 &\le M^{[m]}_{R\psi, \varphi}\|f\|_{\mathcal B_\nu(B_X)}.
\endaligned$$
Consequently, by (\ref{eq_Prop4.1}) and $ C_{\psi,\varphi}(0) =0, $ we have
	\begin{equation}\label{eq_Thm6.1a}
\|C_{\psi, \varphi}(f)\|_{\mathcal B_\mu(B_X)} =\sup_{x \in OS_m}\|(C_{\psi,\varphi}(f))_x\|_{\mathcal{B}_{\mu^{[m]}}(\BB_m)} 
\le M^{[m]}_{R\psi, \varphi} \|f\|_{\mathcal B_\nu(B_X)}.
\end{equation}
Thus, (5) is proved.  

\medskip
(3) $\Rightarrow$ (5): Use  an argument analogous and an estimate to the previous one.

\medskip
(6) $\Rightarrow$ (4): Suppose $C^0_{\psi, \varphi}: \mathcal B_{\nu,0}(B_X) \to \mathcal B_{\mu}(B_X)$ is bounded.

First, it is obvious that $\psi \in \mathcal B_\mu(B_X)$ because $ 1 \in \mathcal B_\mu(B_X) $ and  
\begin{equation}\label{psi_0}
	\psi(z)  = \psi(0) + \int_0^1R\psi(tz)\dfrac{dt}{t}= \psi(0) + (C_{\psi,\varphi}1)(z).  
	\end{equation}
Fix $ z \in B_X, $ put $ w = \varphi(z) $ and consider the test function $ \beta_w \in \mathcal B_{\nu,0}(B_X)$ defined by (\ref{test_func2}). Then, by (\ref{nabla_1})
\begin{equation}\label{eq_6.2a}
	\begin{aligned}
		\|C_{\psi,\varphi}\beta_w\|_{\mathcal B_\mu(B_X)}&= \mu(z) \|R(\beta_w(\varphi(z))\| = \mu(z)|\beta_w(\varphi(z))|  \|R\psi(z)\| \\
		&= \mu(z)\|R\psi(z)\|\int_0^{\|\varphi(z)\|^2}g(t)dt \le \|\beta_w\|_{\mathcal B_\nu(B_X)}\|C^0_{\psi,\varphi}\| \\
		&\le C_2 \|C^0_{\psi,\varphi}\| <\infty.
	\end{aligned}
\end{equation}
Let $ r_1 = \nu^{-1}(1/2).$ 

If $ \|w\| \ge r_1 $ by (\ref{test_g2}), (\ref{test_g2a}) and (\ref{eq_6.2a}) we have
\[ \begin{aligned}
	\mu(z)\|R\psi(z)\|\int_0^{\|w\|}\frac{dt}{\nu(t)} &\le 
	 \mu(z)\|R\psi(z)\|\int_0^{\|w\|}\frac{g(t)}{C_1}dt \\
	&\le \frac{C_3}{C_1}\mu(z)\|R\psi(z)\|\int_0^{\|w\|^2}g(t)dt \\
	&\le \frac{C_3C_2}{C_1}\|C^0_{\psi,\varphi}\| < \infty.
\end{aligned} \]
If $ \|w\| < r_1 $ by (\ref{test_g2})  again  we obtain
\[ \begin{aligned}
	\nu(z)\|R\psi(z)\|\int_0^{\|w\|}\frac{dt}{\nu(t)}
	&\le \mu(z)\|R\psi(z)\|\int_0^{\|w\|}\frac{g(t)}{C_1}dt \\
	&\le \frac{1}{C_1}\mu(z)\|R\psi(z)\| = \frac{1}{C_1}\|C^0_{\psi,\varphi}1\|_{\mathcal B_\mu(B_X)}  \\
	&\le \frac{1}{C_1}\|C^0_{\psi,\varphi}\| < \infty.
\end{aligned} \]
Hence we get $ M_{R\psi,\varphi} <\infty. $

The proof of Theorem is completed.
\end{proof}

\subsection{Proof of Theorem \ref{thm2_2}}
\begin{proof}[Proof of Theorem \ref{thm2_2}] As in the previous one, it is obvious that (4) $ \Rightarrow$ (2) $ \Rightarrow $ (1) and (4) $ \Rightarrow$ (3). 
	
	(1) $ \Rightarrow $  (5):  Assume that $ \psi \in \mathcal B_{\mu,0} $  and  $M^{[m]}_{R\psi,\varphi} < \infty$  for some $ m\ge 2. $ By Theorem \ref{thm6_1}, $ C^0_{\psi,\varphi}: \mathcal B_{\nu,0}(B_X) \to \mathcal B_{\mu}(B_X) $ is bounded.  It suffices to show that $ C^{0,0}_{\psi,\varphi}f \in \mathcal B_{\mu,0}(B_X) $ for every $ f \in \mathcal B_{\nu,0}(B_X). $
	
	Let $ f \in \mathcal B_{\nu,0}(B_X) $ be arbitrarily fixed. Let $ \varepsilon >0 $ be fixed. Then there exists $ r_0 \in (1/2,1) $ such that
	\begin{equation}\label{eq_6.7a} \nu(z)|Rf(z)| < \frac{\varepsilon}{4M_{R\psi,\varphi}}, \quad  \|z\| \ge r_0. 
	\end{equation}
By (\ref{eq_5.1}) we have
	\begin{equation}\label{eq_6.01}
K := \sup_{\|w\| \le r_0} |f(w)| <\infty.
	\end{equation}
	Since  $\psi^{[m]}  \in \mathcal B_{\mu^{[m]},0}(\BB_m)$  we can find $ \theta \in (0,1)$ such that for every $ x\in OS_m $
	\begin{equation}\label{eq_6.00}
		\mu^{[m]}(z_x)\|R\psi^{[m]}(z_x)\| < \dfrac{\varepsilon}{2K} \quad\text{whenever}\   \theta < \|z_x\| < 1. 
	\end{equation}

	For $ \theta < \|z_x\| < 1 $ we consider two cases:
	
	$ \bullet $ The case $ \|w^{[m]}\| := \|\varphi^{[m]}(z_x)\| >r_0: $ Let $ \widehat w^{[m]} =  r_0\frac{w^{[m]}}{\|w^{[m]}\|}.$ We have
	\[ \begin{aligned}
		|f(\varphi^{[m]}(z_x)) - f(\widehat w^{[m]})| &=	|f(w^{[m]}) - f(\widehat w^{[m]})|  \le \int_{r_0/\|w^{[m]}\|}^1\frac{\|Rf(tw^{[m]})\|}{t}dt \\
		&\le \frac{\|w^{[m]}\|}{r_0}\int_{r_0/\|w^{[m]}\|}^1\|Rf(tw^{[m]})\|dt \\
		&\le  \frac{\varepsilon\|w^{[m]}\|}{4M^{[m]}_{R\psi,\varphi} r_0}\int_{r_0/\|w^{[m]}\|}^1\frac{1}{\nu(t\|w^{[m]})\|}dt \\
		&\le \frac{\varepsilon}{2M^{[m]}_{R\psi,\varphi}}\int_{r_0}^{\|w^{[m]}\|}\frac{1}{\nu(t)}dt.
	\end{aligned} \]
	Combining (\ref{eq_6.7a}) with (\ref{eq_6.00}), 
	and by (\ref{nabla_1}),
	for $ \theta < \|z_x\| < 1 $ we obtain
	\[\begin{aligned}
		\mu^{[m]}(z_x)&\|RC^{0,0}_{\psi^{[m]},\varphi^{[m]}}f(z_x)\| = \mu^{[m]}(z_x)|f(\varphi^{[m]}(z_x))| \|R\psi^{[m]}(z_x)\| \\
		&\le \mu^{[m]}(z)\|R\psi^{[m]}(z_x)\| |f(w^{[m]} - f(\widehat w^{[m]})|+ \mu^{[m]}(z_x)\|R\psi^{[m]}(z_x)\| |f(\widehat w_{(m)})| \\
		&\le  M^{[m]}_{R\psi,\varphi}\frac{\varepsilon}{2M^{[m]}_{R\psi,\varphi}} + K\frac{\varepsilon}{2K} =\varepsilon.
	\end{aligned}\]
	
	$ \bullet $ The case $\|w^{[m]}\|:= \|\varphi^{[m]}(z)\| \le r_0: $ We have
	\[
	\mu^{[m]}(z_x)\|R(C^{0,0}_{\psi^{[m]},\varphi^{[m]}}f(z_x))\|  = \mu^{[m]}(z_x)|f(\varphi^{[m]}(z_x))| \|R\psi^{[m]}(z_x)\| < K\frac{\varepsilon}{2K} < \varepsilon.
	\]
	Consequently, by (\ref{eq_Prop4.1}) we have
	\[ \mu(z)\|R(C^{0,0}_{\psi,\varphi}f(z))\| =  \sup_{x \in OS_m}\mu^{[m]}(z_x)\|R(C^{0,0}_{\psi^{[m]},\varphi^{[m]}}f(z_x))\|  \le \varepsilon  \quad\text{whenever}\   \theta < \|z\| < 1. \]
	Thus, $ C^{0,0}_{\psi,\varphi}f \in \mathcal B_{\mu,0}(B_X). $  
	
		(3) $ \Rightarrow $ (5):  
		By an estimate analogous to that used for the proof of (1) $ \Rightarrow $ (5) 		 we get $ C^{0,0}_{\psi,\varphi_{(k)}}f \in \mathcal B_\mu(B_X). $ Then, since
		\[ 	\mu(z)\|RC^{0,0}_{\psi,\varphi}f(z)\| = \lim_{m\to \infty}\mu(z)\|RC^{0,0}_{\psi,\varphi_{(k)}}f(z)\|  \]
		and  $ \mathcal B_{\mu,0}(B_X) $ is closed in $ \mathcal B_{\mu}(B_X), $	we infer that    $ C^{0,0}_{\psi,\varphi}(f) \in \mathcal B_{\mu,0}(B_X). $ 

	(5) $ \Rightarrow $ (4): 
	Assume that $ C^{0,0}_{\psi,\varphi} $ is bounded. By (\ref{psi_0}), $ \psi \in \mathcal B_{\mu,0}(B_X). $ And then, we obtain $ M_{R\psi,\varphi}<\infty $   by the proof of Theorem \ref{thm6_1}.
	%
\end{proof}

\section{Proof of Theorems \ref{thm_compact}, \ref{thm_6.4}, \ref{thm_compact2} and \ref{thm2_2_cp}}
In order to study the compactness of  the operators  $W_{\psi. \varphi},$   as in \cite{Tj} we can prove the following.
\begin{lem}\label{lem_Tj}
	Let $E, F$ be two Banach spaces of holomorphic functions on $B_X.$ Suppose 
	that
	\begin{enumerate}[\rm(1)]
		\item The point evaluation functionals on $E$ are continuous;
		\item The closed unit ball of $E$ is a compact subset of $E$ in the topology of uniform
		convergence on compact sets;
		\item  $T : E \to F$ is continuous when $E$ and $F$ are given the topology of uniform
		convergence on compact sets.
	\end{enumerate}
	Then, $T$ is a compact operator if and only if given a bounded sequence $\{f_n\}$ in $ E $ such
	that $f_n \to 0$ uniformly on compact sets, then the sequence $\{Tf_n\}$ converges to zero in the norm of $F.$
\end{lem}

We can now combine this result with  Montel theorem and  (\ref{eq_5.1}) to  obtain the following proposition. The details of the proof are omitted here.
\begin{prop}\label{prop_C_compact} Let $\psi \in H(B_X)$ and $\varphi\in S(B_X).$   Then  $W_{\psi, \varphi}: \mathcal B_{\nu}(B_X) \to \mathcal B_{\mu}(B_X)$ is compact if and only if $\|C_{\psi, \varphi}(f_n)\|_{\mathcal B_\mu(B_X)} \to 0$ for any bounded sequence $\{f_n\}$ in $\mathcal B_{\nu}(B_X)$ converging to $0$ uniformly on  compact sets in $B_X.$ 
\end{prop}

\subsection{Proof of Theorem \ref{thm_compact}}
\begin{proof}[Proof of Theorem \ref{thm_compact}] First, we prove (B).
	
\medspace
(B) 	The implications (6) $\Rightarrow$  (5)  $\Rightarrow$  (4) and   (2) $\Rightarrow$  (3)  are  obvious. 
	
	(4) $\Rightarrow$ (2): For each $ x \in OS_m $ we write $ z_x := \sum_{k=1}^mz_kx_k. $  Denote $ w^{[m]} = \varphi^{[m]}(z_x).$	 Since $ \int_0^1\frac{dt}{\nu(t)} = \infty $ and $ (\ref{B_1})\to0 $ as $ \|\varphi^{[m]}(y)\| \to 1 $, $ \psi^{[m]} \in \mathcal B_{\mu,0}(\BB_m) $ and
	for any $\varepsilon > 0,$ there exists $ r_0 \in (1/2, 1) $ 	such that 
	\begin{equation}\label{limit_2compact_a} 	 \mu^{[m]}(z_x)\|R\psi^{[m]}(z_x)\|\int_{r_0}^{\|w^{[m]}\|}\frac{dt}{\nu^{[m]}(t)} <\frac{\varepsilon}{3/C_2}\quad \text{for} \  r_0 < \|w^{[m]}\| < 1. \end{equation}
	
%
	
	Let $\{f_n\}_{n\ge1}$ be a bounded sequence in $\mathcal B_{\nu}(B_X)$ converging to $0$ uniformly on compact subsets of $B_X$ and fix an $\varepsilon > 0.$ We may assume that $ \|f_n\|_{\mathcal B_\nu(B_X)} \le 1 $ for every $ n\ge 1. $ 
	By the hypothesis on the sequence $ \{f_n\}_{n\ge1}, $ 
	 there exists a positive integer $ n_0 $ such that
\[	 		 |f_n(w)| \le \frac{\varepsilon}{3\|\psi^{[m]}\|_{\mathcal B_{\mu^{[m]}}(\BB_m)}}, \quad n \ge n_0,\  w \in \overline{B[\varphi^{[m]},r_0]}.\]
	Therefore, for every $ n\ge n_0 $ and  $ \|w^{[m]}\| \le r_0  $ we have
	\begin{equation}\label{limit_2compact_b} 	
		\mu^{[m]}(z_x)\|R\psi^{[m]}(z_x)\| |f_n(w^{[m]})| < \frac{\varepsilon}{3}.\end{equation}
	
Now, for every $ n\ge n_0 $ and $ r_0 < \|w^{[m]}\| <1, $ with noting that  $ r_0\frac{w^{[m]}}{\|w^{[m]}\|} \in   \overline{B[\varphi^{[m]},r_0]},$ by (\ref{test_g2}), (\ref{limit_2compact_a}) and (\ref{limit_2compact_b}) we have
\begin{equation}\label{limit_2compact_c} 	
	\begin{aligned}
	\mu^{[m]}(z_x)&\|R\psi^{[m]}(z_x)\| |f_n(w^{[m]})| \\
	&\le 	\mu^{[m]}(z_x)\|R\psi^{[m]}(z_x)\| \Big|f_n(w^{[m]})- f_n\Big(r_0\frac{w^{[m]})}{\|w^{[m]}\|}\Big)\Big| \\
	&\quad + \mu^{[m]}(z_x)\|R\psi^{[m]}(z_x)\| \Big|f_n\Big(r_0\frac{w^{[m]})}{\|w^{[m]}\|}\Big)\Big| \\
	&\le 	\mu^{[m]}(z_x)\|R\psi^{[m]}(z_x)\|\int_{r_0/\|w^{[m]}\|}^1\|Rf_n(tw^{[m]})\|\frac{dt}{t} +\frac{\varepsilon}{3} \\
	&\le \frac{1}{C_2}\mu^{[m]}(z_x)\|R\psi^{[m]}(z_x)\|\frac{\|w^{[m]}\|}{r_0}\int_{r_0/\|w^{[m]}\|}^1\frac{1}{\nu(t\|w^{[m]}\|)}dt  +\frac{\varepsilon}{3} \\
	&\le \frac{1}{C_2}\mu^{[m]}(z_x)\|R\psi^{[m]}(z_x)\|\int_{r_0}^{\|w^{[m]}\|}\frac{1}{\nu(t)}dt +\frac{\varepsilon}{3} \\
	&\le \frac{\varepsilon}{3} + \frac{\varepsilon}{3} = \frac{2\varepsilon}{3}.
	\end{aligned}
	\end{equation}
Consequently, it follows from (\ref{limit_2compact_b}) and  (\ref{limit_2compact_b}) that
\[ \mu^{[m]}(z_x)\|R\psi^{[m]}(z_x)\||f_n(w^{[m]})| < \varepsilon  \]
for every $ z_x \in \BB_m $ and every $ n\ge n_0. $
This means $ \|C_{\psi, \varphi}(f_n)_x\|_{\mathcal B_\mu(B_X)} \to 0 $ as $ n\to\infty. $
	
	Finally, with the  note that the above estimates is independent of $ x \in OS_m, $  we obtain 
	\[  \|C_{\psi,\varphi}f_n\|_{\mathcal B_{\mu}(B_X)} =  \sup_{x\in OS_m} \|C_{\psi, \varphi}(f_n)_x\|_{\mathcal B_\mu(B_X)}\to 0 \quad \text{as} \ n \to \infty.  \]
	
Hence, by Lemma \ref{prop_C_compact}, $ C_{\psi,\varphi} $ is compact.

	\medskip
	(3) $\Rightarrow$ (6): 
	Suppose $C^0_{\psi, \varphi}$ is compact. Then  clearly, $C^0_{\psi, \varphi}$ is bounded.

	Firstly, assume that   $ (\ref{B_4})\not\to0 $ as $ \|\varphi(y)\| \to 1. $	 Then we can take $ \varepsilon_0 > 0 $ and a sequence $\{z^n\}_{n\ge1} \subset  B_X $ such that
	$\|w^n\| := \|\varphi(z^m)\| \to 1$ but
	\begin{equation}\label{es_6.18}
		 \mu(z^n)\|R\psi(z^n)\|\int_0^{\|w^n\|}\frac{dt}{\nu(t)}\ge \varepsilon_0 \quad \text{for every} \  n =1, 2, \ldots   \end{equation}
	We may assume that $ \|z^n\| > r_1 := \nu^{-1}(1/2). $
	
	Consider the   sequence $ \{\gamma_{w^n}\}_{n\ge1} $ defined  by (\ref{test_func3}). By Proposition \ref{lem3}, this sequence is bounded in $\mathcal B_{\nu}(B_X)$ and converges to $0$ uniformly on compact subsets of $ B_X. $  Then $ \|C^0_{\psi,\varphi}\gamma_{w^n}\|_{\mathcal B_\mu(B_X)} \to 0 $ as $ n\to \infty $  by Lemma \ref{prop_C_compact}.   
	
	On the other hand, by (\ref{test_g2}), (\ref{test_g2a}) and (\ref{es_6.18}) we have
	\[ \begin{aligned}
		\|C^0_{\psi,\varphi}\gamma_{w^n}\|_{\mathcal B_\mu(B_X)}&=\sup_{z\in B_X}\mu(z)\|R\psi(z)\| |\gamma_{w^n}(\varphi(z))| \\
		&\ge \mu(z^n)\|R\psi(z^n)\| |\gamma_{w^n}(w^n)| \\
		&=\mu(z^n)\|R\psi(z^n)\|\int_0^{\|w^n\|^2}g(t)dt \\
		&\ge \frac{C_1}{C_3}\mu(z^n)\|R\psi(z^n)\|\int_0^{\|w^n\|}\frac{dt}{\nu(t)} \\
		&\ge \frac{C_1}{C_3}\varepsilon_0.
	\end{aligned} \]
Thus, we get a contradiction.  
	 And the proof of  (B)  is complete. 
	 
	 Finally, we prove (A).
	 
	 \medspace
	 (A)  Since (2) $ \Leftrightarrow $ (3) $ \Leftrightarrow $ (6) and (6) $ \Rightarrow $ (1) is obvious, it suffices to prove (1) $ \Rightarrow $ (2).
	  We obtain this proof  by an estimate analogous to that used for the proof of (4) $ \Rightarrow $ (2) by using the known fact that 
	  \begin{equation}\label{phi_relative_compact_1} 
	  	\begin{aligned}
	  		B[\varphi_{(k)},r_0] &= \{\varphi_{(k)}(z): \ \|\varphi_{(k)}(z)\| \le r_0 \} \subset \BB_k \subset \C^k\\
	  		&\quad\ \ \text{is relatively compact for every $ 0\le r_0<1 $ and $ k\ge1, $ }
	  	\end{aligned} 
	  \end{equation}
  instead of (\ref{phi_relative_compact}).
\end{proof}
 
\subsection{Proof of Theorem \ref{thm_6.4}}
   \begin{proof}[Proof of Theorem \ref{thm_6.4}]
  	For every $ z\in B_X, $ consider  the function $ \delta_z $ given by $ \delta_z(f) = f(z)$ for every $ f \in \mathcal B_\mu(B_X). $ 
  By (\ref{eq_5.1}), it is clear that $ \delta_z \in (\mathcal B_\mu(B_X))'.$ Moreover, we have
  \begin{equation}\label{eq_6.18}
  	\dfrac{1}{2}\|z -w\| \le \|\delta_z - \delta_w\| \quad \forall z,w \in B_X.
  \end{equation} 
  Indeed, it is easy to check by direct calculation that
  \[ \dfrac{1}{2}\|z-w\| \le \sqrt{1-\dfrac{(1-\|z\|^2)(1-\|w\|^2)}{|1-\langle z,w\rangle|^2}} = \varrho_X(z,w) \]
  where $ \varrho_X $ is the pseudohyperbolic metric in $ B_X $ (see \cite[p.99]{GR}). On the other hand, we also have
  \[ \varrho_X(z,w) = \sup\{\varrho(f(z),f(w)):\ \text{$ f \in H^\infty(B_X) $ with $ \|f\|_\infty \le 1 $}\} \]
  (see (3.4) in \cite{BGM}), where $ \varrho(x,y) = \big|\frac{x-y}{1-\overline{x}y}\big|, $  $ x,y \in \BB_1, $ is the pseudohyperbolic metric in $ \BB_1. $ Note that, since the function $ \eta \mapsto \frac{\eta}{1-\overline{f(z)}f(w)} $ is holomorphic from $ \BB_1 $ into $ \BB_1 $ and $ f(z) - f(w) \mapsto 0,$ it follows from Schwarz's lemma that 
  $ \varrho(f(z), f(w)) \le |f(z) - f(w)| $ for every $ z, w \in B_X. $ Consequently,
  \[\begin{aligned}
  	\varrho_X(z,w) &\le \sup\{|f(z) - f(w)|:\ \ \text{for $ f \in H^\infty(B_X) $ with $ \|f\|_\infty \le 1 $}\}  \\
  	&\le \sup\{|\delta_z(f) - \delta_w(f)|:\ \ \text{for $ f \in H^\infty(B_X) $ with $ \|f\|_\infty \le 1 $}\} \\
  	&= \|\delta_z- \delta_w\|.
  \end{aligned}\]
  Hence, (\ref{eq_6.18}) is proved.

  For $ 0<r<1, $ the set $ V_r := \{\delta_z:\ \|z\| \le r\} \subset (\mathcal B_\mu^\nabla(B_X))'$ is bounded. Then, by the compactness of 
  $C_{\psi, \varphi} $ the set
  \[  (C_{\psi,\varphi})^*(V_z) = \{\psi(z)\delta_{\varphi(z)}:\ \|z\|\le r\} \] is relatively compact in $ (H^\infty(B_X))'. $

  It should be noted that, for every  subset  $ K $ of the dual of  a Banach space $ E  $ and every bounded subset $ D \subset \C, $ if the  set $ \{t\eta: \ t \in D, \eta \in A\} $ is relatively compact in $ E  $  then $ A \subset E' $ is relatively compact. With this fact in hand, since the  set $ \{\psi(z):\ \|z\|\le r\} $ is bounded,    the set $  \{\delta_z, \ \|z\|\le r\} $ is relatively compact. Then, it follows from the   inequality (\ref{eq_6.18}) 
  that $ \varphi(rB_X)$  is relatively compact.
   \end{proof}


\subsection{Proof of Theorem \ref{thm_compact2}}
\begin{proof}[Proof of Theorem \ref{thm_compact2}]
Since the implications  (3) $ \Rightarrow $ (2)  $ \Rightarrow $ (1) and  (4) $ \Rightarrow $ (5) are clear, it suffices to prove (1) $ \Rightarrow $ (4) and (5) $\Rightarrow$ (3).
	
	(1) $ \Rightarrow $ (4):  Since $ \int_0^1\frac{dt}{\nu(t)} <\infty $ and $ \psi^{[m]}  \in \mathcal B_{\mu^{[m]}}(\BB_m)$ we have $ M_{R\psi^{[m]},\varphi^{[m]}} <\infty, $ therefore, $ C_{\psi,\varphi} $ is bounded and for any $ \varepsilon >0, $ there exists $ r_0 \in (1/2,1) $ such that  (\ref{limit_2compact_a}) holds. The rest of the proof is similar to the case (1) $ \Rightarrow $ (4) of Theorem \ref{thm6_1}.

	(5) $ \Rightarrow $ (3): Suppose $C^0_{\psi, \varphi}: \mathcal B_{\nu,0}(B_X) \to \mathcal B_{\mu}(B_X)$ iscompact, it is bounded. As in the proof (5) $ \Rightarrow $ (3) of Theorem \ref{thm6_1} we get $ \psi \in \mathcal B_\mu(B_X). $
	
		This  concludes the proof of the theorem.
\end{proof}

	\subsection{Proof of Theorem \ref{thm2_2_cp}}
\begin{proof}[Proof of Theorem \ref{thm2_2_cp}]
	First, we prove (B). 
	
 It is clear that (5) $ \Rightarrow $ (4) $ \Rightarrow $ (3).
	
	\medspace
(3)	$ \Rightarrow $ (2): By Theorems \ref{thm2_2}, \ref{thm_compact}(B) and \ref{thm_compact2} we obtain (2) from (3).

\medspace
(2) $ \Rightarrow $ (5): By the hypothesis and Theorem \ref{thm2_2}, $ \psi \in \mathcal B_{\mu,0}(B_X) .$ Therefore, (5) holds if $ \int_0^1\frac{dt}{\nu(t)}<\infty, $ 
In  the case $ \int_0^1\frac{dt}{\nu(t)}=\infty,$ in the same way as in the proof of (3) $ \Rightarrow $ (6) of Theorem \ref{thm_compact} we also obtain that  (5) holds.

\medspace
Finally, the proof of (A) is similar to the one of (B) but here we use Theorem \ref{thm_compact}(A) instead of Theorem \ref{thm_compact}(B).
 \end{proof}


\end{document}